\documentclass{article}
\usepackage{amssymb}
\usepackage{amsthm}
\usepackage{amsmath}

{\theoremstyle{plain}%
  \newtheorem{theorem}{Theorem}[section]
  \newtheorem{corollary}{Corollary}[section]
  \newtheorem{proposition}{Proposition}[section]
  \newtheorem{lemma}{Lemma}[section]
  \newtheorem{Definition}{Definition}[section]

{\theoremstyle{remark}

}
{\theoremstyle{definition}

}
\newfont{\hueca}{msbm10}
\def\hu #1{\hbox{\hueca #1}}\def\hu #1{\hbox{\hueca #1}}
 \font\got=eufm10 \font\gots=eufm10 at
7pt
\def\g2{ \hbox{\got g}_2}
\def\f4{\hbox{\got f}_4}
\def\e6{\hbox{\got e}_6}
\def\fs4{\hbox{\gots f}_4}
\def\F4{\hbox{\got F}_4}
\def\Fs4{\hbox{\gots F}_4}

\def\C{\mathbb C}

\def\GL{ {\rm GL}}

\def\al{\ifcase\xypolynode\or F \or A\or B\or C\or D\or G\fi}
\def\ala{\ifcase\xypolynode\or a \or b\or c\or d\or g\or f\fi}

\begin{document}

\title{On the Structure of Graded Lie Superalgebras
}
\author{Antonio J. Calder\'{o}n Mart\'{\i }n
\thanks{Supported by the PCI of the UCA `Teor\'\i a de Lie y Teor\'\i a de Espacios de
Banach', by the PAI with project numbers FQM298, FQM2467, FQM3737
and by the project of the Spanish Ministerio de Educaci\'on y
Ciencia MTM2007-60333.}\\
Jos\'{e} M. S\'{a}nchez Delgado \\
Departamento de Matem\'{a}ticas. \\
Universidad de C\'{a}diz. 11510 Puerto Real, C\'{a}diz, Spain.\\
e-mail: ajesus.calderon@uca.es\\
e-mail: josemaria.sanchezdelgado@alum.uca.es }
\date{}
\maketitle

\begin{abstract}
We study the structure of graded Lie superalgebras with arbitrary
dimension and over an arbitrary field ${\hu K}$. We show that any
of such algebras ${\frak L}$ with a symmetric $G$-support is of
the form ${\frak L} = U + \sum\limits_{j}I_{j}$ with $U$ a
subspace of ${\frak L}_1$ and any $I_{j}$ a well described graded
ideal of ${\frak L}$, satisfying $[I_j,I_k] = 0$ if $j\neq k$.
Under certain conditions, it is shown that ${\frak L} =
(\bigoplus\limits_{k \in K} I_k) \oplus (\bigoplus\limits_{q \in
Q} I_q),$ where any $I_k$ is a gr-simple graded ideal of ${\frak
L}$ and any $I_q$ a completely determined low dimensional non
gr-simple graded ideal of ${\frak L}$, satisfying $[I_q,I_{q'}] =
0$ for any $q'\in Q$ with $q \neq q'$.


\end{abstract}

\section{Introduction and previous definitions}

Throughout this paper, Lie superalgebras ${\frak L}$ are
considered of arbitrary dimension and over an arbitrary field
${\hu K}$. It is worth to mention that, unless otherwise stated,
there is not any restriction on $\dim {\frak L}_g$ or the products
$[{\frak L}_g,{\frak L}_{g^{-1}}]$, where ${\frak L}_g$ denotes
the homogeneous subspace associated to $g \in G$. Lie
superalgebras plays an important role in theoretical physics,
specially in conformal field theory and supersymmetries (see
\cite{JMP1, JMP2, JMP3} for recent references). The notion of
supersymmetry reflects the known symmetry between bosons and
fermions, being the mathematical structure formalizing this idea
the one of supergroup, or ${\hu Z}_2$-graded Lie group. As
mentioned in \cite{Cand11}, its job is that of modelling
continuous supersymmetry transformations between bosons and
fermions. As Lie algebras consist of generators of Lie groups, the
infinitesimal Lie group elements tangent to the identity, so ${\hu
Z}_2$--graded Lie algebras, otherwise known as Lie superalgebras,
consist of generators of (or infinitesimal) supersymmetry
transformations. We also refers to    \cite{bena} and \cite{Gie}
for more interesting applications of Lie superalgebras. The
interest on gradings on simple Lie algebras has been remarkable in
the last years. The gradings of finite dimensional simple Lie
algebras, ruling out $\frak a_l$, $\frak d_4$ and the exceptional
cases, are described in \cite{Ivan2}. The fine gradings on $\frak
a_l$ have been determined in \cite{LGII} solving the related
problem of finding maximal abelian groups of diagonalizable
automorphisms of the algebras (not only in $\GL(n,\C)$ but also in
${{\rm O}}(n,\C)$ for $n\ne 8$ and ${{\rm SP}}(2n,\C)$). See also
\cite{feng, Phys1,Phys3, Phys4}. The first studies of gradings on
exceptional Lie algebras are \cite{otrog2}, \cite{g2} and
\cite{f4}, which describe the group gradings on $\g2$ and $\f4$.
The study of the gradings of the real forms of complex Lie
algebras begins in \cite{LGIII}, where are considered the gradings
on the real forms of classical simple complex Lie algebras. The
description of the fine gradings of the real forms of the
exceptional simple Lie algebras ${\frak f}_4$ and ${\frak g}_2$
are obtained in \cite{Reales}. Respect to the group gradings on
superalgebras, these have been considered, for the case of the
Jordan superalgebra $K_{10}$, in \cite{Kac}.

In the present paper we begin the study of arbitrary graded Lie
superalgebras, (not necessarily simple or finite-dimensional), and
over an arbitrary field ${\hu K}$ by focussing on their structure.
In $\S2$ we extend the techniques of connections in the support of
$G$ developed for graded Lie algebras in \cite{graalg} to the
framework of graded Lie superalgebras ${\frak L}$, so as to show
that ${\frak L}$ is of the form ${\frak L} = U + \sum\limits_jI_j$
with $U$ a subspace of ${\frak L}_1$ and any $I_j$ a well
described graded ideal of ${\frak L}$, satisfying $[I_j,I_k] = 0$
if $j \neq k$. In $\S3$, and under certain conditions, it is shown
that ${\frak L} =(\bigoplus\limits_{k \in K} I_k) \oplus
(\bigoplus\limits_{q \in Q} I_q),$ where any $I_k$ is a gr-simple
graded ideal of ${\frak L}$ and any $I_q$ a completely determined
low dimensional non gr-simple graded ideal of ${\frak L}$,
satisfying $[I_q,I_{q'}] = 0$ for any $q' \in Q$ with $q \neq q'$.

\medskip

\medskip

A {\it Lie superalgebra} ${\frak L}$ is a ${\hu Z}_2$-graded
algebra ${\frak L} = {\frak L}^{\bar 0} \oplus {\frak L}^{\bar 1}$
over an arbitrary ground field ${\hu K}$, endowed with a bilinear
product $[\cdot, \cdot]$, whose homogenous elements $x \in {\frak
L}^{\bar i}, y \in {\frak L}^{\bar j}, {\bar i}, {\bar j} \in {\hu
Z}_2$, satisfy
$$[x,y] \subset {\frak L}^{{\bar i}+{\bar j}}$$
$$[x,y] = -(-1)^{{\bar i}{\bar j}}[y,x] \hbox{ (Skew-supersymmetry)}$$
$$[x,[y,z]] = [[x,y],z] + (-1)^{{\bar i}{\bar j}}[y,[x,z]] \hbox{ (Super Jacobi
identity)}$$
for any homogenous element $z \in {\frak L}^{\bar k}$,
${\bar k} \in {\hu Z}_2$.

Note that if ${\rm char}({\hu K}) \neq 2$, then ${\frak L}^{\bar
0}$ is a Lie algebra called the even or bosonic part of ${\frak
L}$ while ${\frak L}^{\bar 1}$ is called the odd or fermonic part
of ${\frak L}$.

\medskip

The term {\it grading} will always mean abelian group grading
compatible with the ${\hu Z}_2$-graduation providing the
superalgebra structure of ${\frak L}$. That is, a decomposition in
vector subspaces $${\frak L} = \bigoplus\limits_{g \in G}{\frak
L}_g$$ where $G$ is an abelian group, and the homogeneous
subspaces satisfy
\begin{equation}\label{separa}
\hbox{${\frak L}_g = {\frak L}_g^{\bar 0} \oplus {\frak L}_g^{\bar
1}$ with ${\frak L}_g^{\bar i} = {\frak L}_g \cap {\frak L}^{\bar
i}, {\bar i} \in {\hu Z}_2$,}
\end{equation}
and $$[{\frak L}_g,{\frak L}_{g'}] \subset {\frak L}_{gg'},$$
(denoting by juxtaposition the product in $G$). We note that a
$G$-grading of ${\frak L}$ provides a refinement of the initial
${\hu Z}_2$-grading of ${\frak L}$ (see \cite[Definition
3.1.4]{Koche}), and that split Lie superalgebras and graded Lie
algebras are examples of graded Lie superalgebras. Hence, the
present paper extends the results in \cite{graalg, Nosalg}.

We call the {\it $G$-support} of the grading to the set
$$\Sigma_G := \{g \in G \setminus \{1\} : {\frak L}_g \neq 0\}.$$
So we can write  $$\hbox{${\frak L} = \bigoplus\limits_{g \in
G}({\frak L}_g^{\bar 0} \oplus {\frak L}_g^{\bar 1}) = ({\frak
L}_1^{\bar 0} \oplus {\frak L}_1^{\bar 1}) \oplus
(\bigoplus\limits_{g \in \Sigma_G}({\frak L}_g^{\bar 0} \oplus
{\frak L}_g^{\bar 1}))$.}$$
We also denote by $\Sigma_G^{\bar 0} :=
\{g \in \Sigma_G : {\frak L}_g^{\bar 0} \neq 0\}$ and by
$\Sigma_G^{\bar 1} := \{g \in \Sigma_G : {\frak L}_g^{\bar 1} \neq
0\}$. So $\Sigma_G = \Sigma_G^{\bar 0} \cup \Sigma_G^{\bar 1}$,
being a non necessarily disjoint union.

The $G$-support $\Sigma_G$ is called {\it symmetric} if $g \in
\Sigma_G^{\bar i}$ implies $g^{-1} \in \Sigma_G^{\bar i}, {\bar i}
\in {\hu Z}_2$.

\medskip

The usual regularity concepts will be understood in the graded
sense, (compatible with the initial ${\hu Z}_2$-graduation of
${\frak L}$). That is, a {\it graded ideal} $I$ of ${\frak L}$ is
an ideal which splits as
\begin{equation}\label{idealpartio}
\hbox{$I = \bigoplus\limits_{g \in G}I_g = \bigoplus\limits_{g \in
G}(I_g^{\bar 0} \oplus I_g^{\bar 1})$ with any $I_g^{\bar i} = I_g
\cap {\frak L}^{\bar i}, {\bar i} \in {\hu Z}_2$.}
\end{equation}
A graded Lie superalgebra ${\frak L}$ will be called {\it
gr-simple} if $[{\frak L},{\frak L}] \neq 0$ and its only graded
ideals are $\{0\}$ and ${\frak L}$.

Observe that from the grading of  ${\frak L}$ and equation
(\ref{separa}) we get
$$[{\frak L}_g^{\bar i}, {\frak L}_{g'}^{\bar j}] \subset
{\frak L}_{gg'}^{{\bar i}+{\bar j}}$$ for any ${\bar i}, {\bar j}
\in {\hu Z}_2$.


\section{Connections in $\Sigma_G$. Decompositions}

From now on and throughout the paper, ${\frak L}$ denotes a graded
Lie superalgebra with a symmetric $G$-support $\Sigma_G$, and
$${\frak L} = \bigoplus\limits_{g \in G}({\frak L}_g^{\bar 0}
\oplus {\frak L}_g^{\bar 1}) = ({\frak L}_1^{\bar 0} \oplus {\frak
L}_1^{\bar 1}) \oplus (\bigoplus\limits_{g \in \Sigma_G}({\frak
L}_g^{\bar 0} \oplus {\frak L}_g^{\bar 1}))$$ the corresponding
grading. We begin by developing connection techniques in this
framework.

\begin{Definition}
Let $g$ and $g'$ be two elements in $\Sigma_G$. We shall say that
$g$ is {\em $\Sigma_G$-connected} to $g'$ if there exist
$g_1,g_2...,g_n \in \Sigma_G$ such that
\begin{enumerate}
\item[{\rm 1.}] $g_1 = g$.
\item[{\rm 2.}]
$\{g_1,g_1g_2,......,g_1g_2 \cdots g_{n-1}\} \subset \Sigma_G.$
\item[{\rm 3.}] $g_1g_2 \cdots g_n \in \{g',{(g')}^{-1}\}$.
\end{enumerate}
We also say that $\{g_1,...,g_n\}$ is a {\em
$\Sigma_G$-connection} from $g$ to $g'$.
\end{Definition}

The next result shows the $\Sigma_G$-connection relation is of
equivalence. Its proof is virtually identical to the proof of
\cite[Proposition 2.1]{graalg} but for completeness reasons we add
an sketch of the same.

\begin{proposition}\label{pro1}
The relation $\sim$ in $\Sigma_G$, defined by $g \sim g'$ if and
only if $g$ is $\Sigma_G$-connected to $g$, is of equivalence.
\end{proposition}

\begin{proof}
$\{g\}$ is a $\Sigma_G$-connection from $g$ to itself and
therefore $g \sim g$.

If $g \sim g'$ and $\{g_1,..., g_n\}$ is a $\Sigma_G$-connection
from $g$ to $g'$, then $$\{g_1 \cdots g_n, g_n^{-1}, g_{n-1}^{-1},
..., g_2^{-1}\}$$ is a $\Sigma_G$-connection from $g'$ to $g$ in
case $g_1 \cdots g_n = g'$, and $$\{g_1^{-1} \cdots g_n^{-1}, g_n,
g_{n-1},...,g_2\}$$ in case $g_1 \cdots g_n = (g')^{-1}$.
Therefore $g' \sim g$.

Finally, suppose $g \sim g'$ and $g' \sim g''$, and write
$\{g_1,..., g_n\}$ for a $\Sigma_G$-connection from $g$ to $g'$
and $\{{g'}_1,..., {g'}_m\}$  for a $\Sigma_G$-connection from
$g'$ to $g''$. If $m > 1$, then
$\{g_1,...,g_n,{g'}_2,...,{g'}_m\}$ is a $\Sigma_G$-connection
from $g$ to $g''$ in case $g_1 \cdots g_n = g'$, and
$\{g_1,...,g_n,{g'}_2^{-1},...,{g'}_m^{-1}\}$ in case $g_1 \cdots
g_n = (g')^{-1}$. If $m = 1$, then $g'' \in \{g', (g')^{-1}\}$ and
so $\{g_1,..., g_n\}$ is a $\Sigma_G$-connection from $g$ to
$g''$. Therefore $g \sim g''$ and $\sim$ is of equivalence.
\end{proof}

Given $g \in \Sigma_g$, we denote by
\[
{\frak C}_{g}:=\{g' \in \Sigma_G : g \hspace{0.1cm} {\rm is}
\hspace{0.1cm} \Sigma_G{\rm-connected} \hspace{0.1cm} {\rm to} \hspace{0.1cm} g'%
\}.
\]

Clearly if $g' \in {\frak C}_g$ then ${(g')}^{-1} \in {\frak C}_g$
and, by Proposition \ref{pro1}, if $h \notin {\frak C}_g$ then
${\frak C}_g \cap {\frak C}_h = \emptyset$.

\begin{lemma}\label{C_g cerrado}
If  $g' \in {\frak C}_g$ and  $g'', g'g'' \in \Sigma_G$, then
$g'', g'g'' \in {\frak C}_g$.
\end{lemma}

\begin{proof}
The $\Sigma_G$-connection $\{g',g''\}$ gives us $g' \sim g'g''$.
Hence,  by the transitivity  of $\sim$, we finally get $g'g'' \in
{\frak C}_g$. To verify $g'' \in {\frak C}_g$, observe that
$\{g'g'', (g')^{-1}\}$ is a $\Sigma_G$-connection from $g'g''$ to
$g''$. Now, taking into account $g'g'' \in {\frak C}_g$, we
conclude as above that $g'' \in {\frak C}_g$.
\end{proof}

Our next goal is to associate an (adequate) graded ideal $I_{[g]}$
to any ${\frak C}_g$. For ${\frak C}_g$, $g \in \Sigma_G$, we
define

$${\frak L}_{{\frak C}_g,1} := span_{\hu
K}\{[{\frak L}_{g'},{\frak L}_{(g')^{-1}}] : g' \in {\frak C}_g\}
=$$
\begin{equation}\label{suma16}
(\sum\limits_{g' \in {\frak C}_g}([{\frak L}_{g'}^{\bar 0},{\frak
L}_{(g')^{-1}}^{\bar 0}] + [{\frak L}_{g'}^{\bar 1},{\frak
L}_{(g')^{-1}}^{\bar 1}]) \oplus (\sum\limits_{g' \in {\frak
C}_g}[{\frak L}_{g'}^{\bar 0},{\frak L}_{(g')^{-1}}^{\bar 1}] +
[{\frak L}_{g'}^{\bar 1},{\frak L}_{(g')^{-1}}^{\bar 0}]))
\end{equation}
$$ \subset {\frak L}_1^{\bar 0} \oplus {\frak L}_1^{\bar 1},$$ last equality
being consequence of equation (\ref{separa}); and
$$V_{{\frak C}_g} := \bigoplus\limits_{g' \in {\frak C}_g}{\frak
L}_{g'} = \bigoplus\limits_{g' \in {\frak C}_g}({\frak
L}_{g'}^{\bar 0} \oplus {\frak L}_{g'}^{\bar 1}).$$ We denote by
${\frak L}_{{\frak C}_g}$ the following (graded) subspace of $L$,
$${\frak L}_{{\frak C}_g} := {\frak L}_{{\frak C}_g,1} \oplus
V_{{\frak C}_g}.$$

\begin{proposition}\label{pro2}
Let $g \in \Sigma_G$. Then the following assertions hold.

\begin{enumerate}
\item[{\rm 1.}] $[{\frak L}_{{\frak C}_g},{\frak L}_{{\frak C}_g}]
\subset {\frak L}_{{\frak C}_g}$.

\item[{\rm 2.}]  If $h \notin {\frak C}_g$ then $[{\frak
L}_{{\frak C}_g},{\frak L}_{{\frak C}_h}] = 0$.
\end{enumerate}
\end{proposition}

\begin{proof}
1. We have
\begin{equation}\label{cuatro2}
[{\frak L}_{{\frak C}_g},{\frak L}_{{\frak C}_g}]=[{\frak
L}_{{\frak C}_g,1} \oplus V_{{\frak C}_g},{\frak L}_{{\frak
C}_g,1} \oplus V_{{\frak C}_g}] \subset
\end{equation}
$$[{\frak L}_{{\frak C}_g,1}, {\frak L}_{{\frak C}_g,1}] + [{\frak L}_{{\frak C}_g,1},
V_{{\frak C}_g}] + [V_{{\frak C}_g},{\frak L}_{{\frak C}_g,1}] +
[V_{{\frak C}_g}, V_{{\frak C}_g}].$$ Consider the above second
summand $[{\frak L}_{{\frak C}_g,1}, V_{{\frak C}_g}]$. Taking
into account ${\frak L}_{{\frak C}_g,1} \subset {\frak L}_1$ and
$[{\frak L}_1,{\frak L}_g] \subset {\frak L}_g$ for any $g \in
\Sigma_g$, we have $[{\frak L}_{{\frak C}_g,1} , V_{{\frak C}_g}]
\subset V_{{\frak C}_g}$. In a similar way $[V_{{\frak
C}_g},{\frak L}_{{\frak C}_g,1}] \subset V_{{\frak C}_g}$ and so
\begin{equation}\label{diez2}
[{\frak L}_{{\frak C}_g,1}, \oplus V_{{\frak C}_g}] + [V_{{\frak
C}_g},{\frak L}_{{\frak C}_g,1}] \subset V_{{\frak C}_g}.
\end{equation}
Consider now the fourth summand $[V_{{\frak C}_g}, V_{{\frak
C}_g}]$ in equation (\ref{cuatro2}) and suppose there exist $g',
g'' \in {\frak C}_g$ such that $[{\frak L}_{g'},{\frak L}_{g''}]
\neq 0$. If $g'' = (g')^{-1},$ clearly $[{\frak L}_{g'},{\frak
L}_{g''}] = [{\frak L}_{g'},{\frak L}_{(g')^{-1}}] \subset {\frak
L}_{{\frak C}_g,1}$. Otherwise, if $g'' \neq (g')^{-1}$, then
$g'g'' \in \Sigma_G$ and Lemma \ref{C_g cerrado} gives us $g'g''
\in {\frak C}_g$. Hence, $[{\frak L}_{g'},{\frak L}_{g''}] \subset
{\frak L}_{g'g''} \subset V_{{\frak C}_g}$. In any case
\begin{equation}\label{nueve2}
[V_{{\frak C}_g}, V_{{\frak C}_g}] \subset {\frak L}_{{\frak
C}_g}.
\end{equation}
Finally, let us consider the first summand $[{\frak L}_{{\frak
C}_g,1},{\frak L}_{{\frak C}_g,1}]$ in equation (\ref{cuatro2}).
We have $$[{\frak L}_{{\frak C}_g,1} ,{\frak L}_{{\frak C}_g,1}] =
[\sum\limits_{g' \in {\frak C}_g}[{\frak L}_{g'},{\frak
L}_{(g')^{-1}}],\sum\limits_{g'' \in {\frak C}_g}[{\frak
L}_{g''},{\frak L}_{(g'')^{-1}}]] \subset$$ $$\sum\limits_{{\tiny
\begin{array}{c}
  g',g'' \in {\frak C}_g \\
 \bar{i},\bar{j},\bar{k},\bar{l} \in {\hu Z}_2 \\
\end{array}}}[[{\frak L}_{g'}^{\bar i},{\frak L}_{(g')^{-1}}^{\bar j}],
[{\frak L}_{g''}^{\bar k},{\frak L}_{(g'')^{-1}}^{\bar l}]],$$
last equality being consequence of equation (\ref{suma16}). Taking
now into account super Jacobi identity we get
$$\sum\limits_{{\tiny
\begin{array}{c}
  g',g'' \in {\frak C}_g \\
 \bar{i},\bar{j},\bar{k},\bar{l} \in {\hu Z}_2 \\
\end{array}}}[[{\frak L}_{g'}^{\bar i},{\frak L}_{(g')^{-1}}^{\bar j}],
[{\frak L}_{g''}^{\bar k},{\frak L}_{(g'')^{-1}}^{\bar l}]]
\subset $$
$$ \sum\limits_{{\tiny
\begin{array}{c}
  g',g'' \in {\frak C}_g \\
 \bar{i},\bar{j},\bar{k},\bar{l} \in {\hu Z}_2 \\
\end{array}}}([{\frak L}_{g'}^{\bar i},[{\frak L}_{(g')^{-1}}^{\bar j},
[{\frak L}_{g''}^{\bar k},{\frak L}_{(g'')^{-1}}^{\bar l}]]] +
[{\frak L}_{(g')^{-1}}^{\bar j},[{\frak L}_{g'}^{\bar i}, [{\frak
L}_{g''}^{\bar k},{\frak L}_{(g'')^{-1}}^{\bar l}]]]) \subset $$
$$\sum\limits_{
  g' \in {\frak C}_g} ([{\frak L}_{g'}^{\bar i},
{\frak L}_{(g')^{-1}}^{{\bar j}+{\bar k}+{\bar l}}]+ [{\frak
L}_{(g')^{-1}}^{{\bar j}},{\frak L}_{g'}^{{\bar i}+{\bar k}+{\bar
l}}] ) \subset \sum\limits_{
  g' \in {\frak C}_g} [{\frak L}_{g'},
{\frak L}_{(g')^{-1}}]={\frak L}_{{\frak C}_g,1}.$$ That is,
\begin{equation}\label{33}
[{\frak L}_{{\frak C}_g,1} ,{\frak L}_{{\frak C}_g,1} ] \subset
{\frak L}_{{\frak C}_g,1}
\end{equation}
From equations (\ref{cuatro2})-(\ref{33}) we conclude $[{\frak
L}_{{\frak C}_g},{\frak L}_{{\frak C}_g}] \subset {\frak
L}_{{\frak C}_g}$.

2. We have as in 1. that
\begin{equation}\label{cuatro22}
[{\frak L}_{{\frak C}_g},{\frak L}_{{\frak C}_h}] \subset [{\frak
L}_{{\frak C}_g,1} ,{\frak L}_{{\frak C}_h,1} ]+ [{\frak
L}_{{\frak C}_g,1} ,  V_{{\frak C}_h}]+ [ V_{{\frak C}_g},{\frak
L}_{{\frak C}_h,1} ]+ [V_{{\frak C}_g}, V_{{\frak C}_h}].
\end{equation}

 Let us suppose that there exist $g^{\prime }  \in {\frak
C}_{g}$ and $h^{\prime }  \in {\frak C}_{h}$ such that $[{\frak
L}_{g^{\prime } },{\frak L}_{h^{\prime } }]\neq 0$. Then
$g^{\prime} h^{\prime } \in \Sigma_G$ and we have as consequence
of Lemma \ref{C_g cerrado}  that $g $ is connected to $h $, a
contradiction. From here $[V_{{\frak C}_g}, V_{{\frak C}_h}]=0$.
Taking into account this equality and the fact $(g^{\prime})^{-1}
\in {\frak C}_{g}$ for any $g^{\prime} \in {\frak C}_{g}$, we can
argue  with super Jacobi identity  in $[[{\frak L}_{g^{\prime }
},{\frak L}_{(g^{\prime })^{-1} }],{\frak L}_{h' }]$, in a similar
way to item 1., to get $[[{\frak L}_{g^{\prime } },{\frak
L}_{(g^{\prime
 })^{-1}}], {\frak L}_{h^{\prime  }}]= 0$. Now a same argument can
 be applied to verify
$[ [{\frak L}_{g^{\prime} },{\frak L}_{(g^{\prime})^{-1}
}],[{\frak L}_{h^{\prime}}, {\frak L}_{(h^{\prime} )^{-1}}]]=0.$
 Taking into account equation (\ref{cuatro22}) we have proved 2.
\end{proof}

\medskip

Proposition \ref{pro2}-1 let us assert that for any $g \in
\Sigma_g$, ${\frak L}_{{\frak C}_g}$ is a (graded) subalgebra of
${\frak L}$ that we call the subalgebra of ${\frak L}$ {\it
associated} to ${\frak C}_g$.

\begin{theorem}\label{teo1}
The following assertions hold.
\begin{enumerate}
\item[{\rm 1.}]  For any $g \in \Sigma_G$, the graded subalgebra
${\frak L}_{{\frak C}_g}={\frak L}_{{\frak C}_g,1} \oplus
V_{{\frak C}_g}$ of ${\frak L}$ associated to ${\frak C}_{g}$ is a
graded ideal of ${\frak L}$.

\item[{\rm 2.}]  If ${\frak L}$ is gr-simple, then there exists a
$\Sigma_G$-connection from $g$ to $g'$ for any $g, g' \in
\Sigma_G$, and ${\frak L}_1 = \sum\limits_{g \in \Sigma_G}[{\frak
L}_g,{\frak L}_{g^{-1}}]$.
\end{enumerate}
\end{theorem}

\begin{proof}
1. Taking into account  Proposition \ref{pro2} we have  $$[{\frak
L}_{{\frak C}_g},{\frak L}] = [{\frak L}_{{\frak C}_g},{\frak L}_1
\oplus (\bigoplus\limits_{g' \in {\frak C}_g}{\frak L}_{g'})
\oplus (\bigoplus\limits_{h \notin {\frak C}_g}{\frak L}_h)\rbrack
=$$
$$= [{\frak L}_{{\frak C}_g},{\frak L}_1] \oplus
(\bigoplus\limits_{g' \in {\frak C}_g}[{\frak L}_{{\frak
C}_g},{\frak L}_{g'}]) \oplus (\bigoplus\limits_{h \notin {\frak
C}_g}[{\frak L}_{{\frak C}_g},{\frak L}_h]) \subset {\frak
L}_{{\frak C}_g}.$$

2. The gr-simplicity of ${\frak L}$ implies ${\frak L}_{{\frak
C}_g} = {\frak L}$. From here ${\frak C}_g = \Sigma_G$ and ${\frak
L}_1 = \sum\limits_{g \in \Sigma_G}[{\frak L}_g,{\frak
L}_{g^{-1}}]$.
\end{proof}

\begin{theorem}\label{teo2}
For the complementary  vector space  $\mathcal{U}$ of $span_{\hu
K}\{[{\frak L}_g,{\frak L}_{g^{-1}}] : g \in \Sigma_G\}$ in
${\frak L}_1$, we have
$${\frak L} = \mathcal{U} + \sum\limits_{g \in \Sigma_G/\sim} I_{[g]},$$ where
any $I_{[g]}$ is one of the graded ideals ${\frak L}_{{\frak
C}_g}$ of ${\frak L}$ described in Theorem \ref{teo1}, satisfying
$[I_{[g]},I_{[g']}]=0$ if $[g] \neq [g']$.
\end{theorem}

\begin{proof}
By Proposition \ref{pro1}, we can consider the quotient set
$\Sigma_G/\sim:= \{[g]:g \in \Sigma_G\}$. Let us denote by
$I_{[g]}:={\frak L}_{{\frak C}_g}.$ We have $I_{[g]}$ is well
defined and, by Theorem \ref{teo1}-1,  a graded ideal of ${\frak
L}$. Therefore
$${\frak L} = \mathcal{U} + \sum\limits_{[g] \in \Sigma_G/\sim} I_{[g]}.$$
By applying Proposition \ref{pro2}-2 we also obtain
$[I_{[g]},I_{[g']}] = 0$ if $[g] \neq [g']$.
\end{proof}

\medskip

Let us denote by ${\mathcal Z}({\frak L})=\{v\in {\frak L}:
[v,{\frak L}]=0\}$ the {\it center} of ${\frak L}$.

\begin{corollary}\label{co1}
If ${\mathcal Z}({\frak L}) = 0$ and ${\frak L}_1 = \sum\limits_{g
\in \Sigma_G}[{\frak L}_g,{\frak L}_{g^{-1}}]$, then ${\frak L}$
is the direct sum of the graded ideals given in Theorem
\ref{teo1}-1,
$${\frak L} = \bigoplus\limits_{[g] \in \Sigma_G/\sim} I_{[g]},$$
which satisfy $[I_{[g]},I_{[g']}]=0$ if $[g] \neq [g']$.
\end{corollary}

\begin{proof}
From ${\frak L}_1 = \sum\limits_{g \in \Sigma_G}[{\frak
L}_g,{\frak L}_{g^{-1}}]$ it is clear that ${\frak L} =
\sum\limits_{[g] \in \Sigma_G/\sim}I_{[g]}$. The direct character
of the sum now follows from the facts $[I_{[g]},I_{[g']}] = 0$, if
$[g] \neq [g']$, and ${\mathcal Z}({\frak L}) = 0$.
\end{proof}

\section{The gr-simple components}

The  study of the structure of this kind of algebras has been
reduced to consider those satisfying that the $G$-support has all
of its elements $\Sigma_G$-connected. It is a natural question if
these algebras are gr-simple. We study this problem in this
section.

\begin{lemma}\label{lema4}
Let ${\frak L}$ be a graded Lie superalgebra with ${\mathcal
Z}({\frak L}) = 0$ and ${\frak L}_1 = \sum\limits_{g \in
\Sigma_G}[{\frak L}_g,{\frak L}_{g^{-1}}]$. If $I$ is a graded
ideal of ${\frak L}$ such that $I \subset {\frak L}_1$, then $I =
\{0\}.$
\end{lemma}

\begin{proof}
Suppose there exists a nonzero graded ideal $I$ of ${\frak L}$
such that $I \subset {\frak L}_1$. We have $[I,
\bigoplus\limits_{g \in \Sigma_G}{\frak L}_g] \subset I \subset
{\frak L}_1$, therefore $[I, \bigoplus\limits_{g \in
\Sigma_G}{\frak L}_g] \subset {\frak L}_1 \cap
(\bigoplus\limits_{g \in \Sigma_G}{\frak L}_g) = 0$. The fact
${\mathcal Z}({\frak L}) = 0$ implies $[I, {\frak L}_1] \neq 0$.
Taking account ${\frak L}_1 = \sum\limits_{g \in \Sigma_G}[{\frak
L}_g,{\frak L}_{g^{-1}}]$, there exists $g_0 \in \Sigma_G$ such
that $[I, [{\frak L}_{g_0},{\frak L}_{g_0^{-1}}]] \neq 0$. By
writing $I=I^{\bar 0} \oplus I^{\bar 1}$ with $I^{\bar i}= I \cap
{\frak L}^{\bar i}$, $i \in {\hu Z}_2$, and taking into account
equation (\ref{separa}) we have $[I^{\bar i}, [{\frak
L}_{g_0}^{\bar j},{\frak L}_{g_0^{-1}}^{\bar k}]] \neq 0$ for some
${\bar i},{\bar j},{\bar k} \in {\hu Z}_2$. Super Jacobi identity
gives us now that either $0 \neq [I^{\bar i}, {\frak
L}_{g_0}^{\bar j}] \subset {\frak L}_{g_0} \cap {\frak L}_1$ or $0
\neq [I^{\bar i},{\frak L}_{g_0^{-1}}^{\bar k}] \subset {\frak
L}_{g_0^{-1}} \cap {\frak L}_1$, a contradiction. Therefore $I =
\{0\}.$
\end{proof}


Let us introduce the concepts of $\Sigma_G$-multiplicativity and
maximal length in the framework of graded Lie superalgebras, in a
similar way to the ones for graded Lie algebras \cite{graalg},
split Lie superalgebras \cite{Nosalg}, and  split Lie triple
systems \cite{triples} among other contexts.  Recall that we
denote by $\Sigma_G^{\bar 0} := \{g \in \Sigma_G : {\frak
L}_g^{\bar 0} \neq 0\}$ and by $\Sigma_G^{\bar 1} := \{g \in
\Sigma_G : {\frak L}_g^{\bar 1} \neq 0\}$.

\begin{Definition}
We say that a graded Lie superalgebra ${\frak L}$ is of {\rm
maximal length} if $\dim {\frak L}_g^{\bar i} \in \{0,1\}$ for any
$g \in \Sigma_G$ and ${\bar i} \in {\hu Z}_2$.
\end{Definition}

Observe that, for a graded Lie superalgebra of maximal length
${\frak L}$, the symmetry of $\Sigma_G$ gives us that given some
$g \in \Sigma_G$,
\begin{equation}\label{dimen1o2}
\hbox{ either ${\rm dim}({\frak L}_g)={\rm dim}({\frak
L}_{g^{-1}})=1$ in the case $g \notin \Sigma_G^{\bar 0} \cap
\Sigma_G^{\bar 1}$, or}
\end{equation}
$$\hbox{${\rm dim}({\frak L}_g)={\rm dim}({\frak L}_{g^{-1}})=2$
in the case  $g \in \Sigma_G^{\bar 0} \cap \Sigma_G^{\bar 1}$.}$$
 Also observe  that equation
(\ref{idealpartio})
let us assert that any nonzero graded ideal $I$ of ${\frak L}$ is
of the form
\begin{equation}\label{max}
I = (I_1^{\bar 0} \oplus (\bigoplus\limits_{g \in \Sigma_I^{\bar
0}} {\frak L}_g^{\bar 0})) \oplus (I_1^{\bar 1} \oplus
(\bigoplus\limits_{g' \in \Sigma_I^{\bar 1}} {\frak L}_{g'}^{\bar
1}))
\end{equation}
${\rm where} \hspace{0.1cm} \Sigma_I^{\bar i} = \{h \in
\Sigma_G^{\bar i} : I \cap {\frak L}_h^{\bar i} \neq 0\}, {\bar i}
\in {\hu Z}_2$.

\begin{Definition}
We say that a graded Lie superalgebra ${\frak L}$ is {\rm
$\Sigma_G$-multiplicative} if given $g \in \Sigma_G^{\bar i}$ and
$g' \in \Sigma_G^{\bar j}, {\bar i}, {\bar j} \in {\hu Z}_2$, such
that $gg' \in \Sigma_G$, then $[{\frak L}_g^{\bar i}, {\frak
L}_{g'}^{\bar j}] \neq 0$.
\end{Definition}

As examples of $\Sigma_G$-multiplicative graded Lie superalgebras
with maximal length we have the split Lie superalgebras considered
in \cite{Nosalg}, the graded Lie algebras in \cite{graalg}, the
semisimple separable $L^*$-algebras \cite{Schue2} and the
semisimple locally finite split Lie algebras over a field of
characteristic zero \cite{Stumme}.


\begin{lemma}\label{dimensiones}
Let ${\frak L}$ be $\Sigma_G$-multiplicative and  of maximal
length. If $\Sigma_G$ has all of its elements
$\Sigma_G$-connected, then either ${\rm dim}({\frak L}_g)=1$ for
any $g \in \Sigma_G$ or ${\rm dim}({\frak L}_g)=2$ for any $g \in
\Sigma_G$.
\end{lemma}
\begin{proof}
Suppose there exists $g \in \Sigma_G$ such that ${\rm dim}({\frak
L}_g)=2$. Hence, ${\rm dim}({\frak L}_{g^{-1}})=2$ and we can
write ${\frak L}_g={\frak L}_g^{\bar 0} \oplus {\frak L}_g^{\bar
1}$ with any ${\frak L}_g^{\bar i}\neq 0$. Given now any $g'\in
\Sigma_G \setminus \{g, g^{-1}\}$, the fact that $g$ and $g'$ are
$\Sigma_G$-connected gives us a $\Sigma_G$-connection
$\{g_1,g_2,....,g_r\}$ from $g$ to $g'$ such that
$$g_1 = g,$$
$$g_1g_2, g_1g_2g_3,..., g_1g_2g_3 \cdots g_{r-1} \in \Sigma_G$$ and
$$g_1g_2g_3 \cdots g_r \in \{g',(g')^{-1}\}.$$
Consider $g_1, g_2$ and $g_1g_2$. Since $g_2 \in \Sigma_G$, some
${\frak L}_{g_2}^{\bar i_2} \neq 0$ with $ {\bar i_2} \in {\hu
Z}_2$, and so $g_2 \in \Sigma_G^{\bar i_2}$. We have $g_1 = g \in
\Sigma_G^{\bar 0}$ and $g_2 \in \Sigma_G^{\bar i_2}$ such that
$g_1g_2 \in \Sigma_G$. Then the $\Sigma_G$-multiplicativity of
${\frak L}$ gives us
\begin{equation}\label{paraaqui}
0 \neq [{\frak L}_{g_1}^{\bar 0}, {\frak L}_{g_2}^{\bar i_2}]
\subset {\frak L}_{g_1g_2}^{\bar i_2}
\end{equation}
Hence, the maximal length of ${\frak L}$ shows $\dim {\frak
L}_{g_1g_2}^{ \bar i_2} = 1$ and so $0 \neq [{\frak L}_{g_1}^{\bar
0}, {\frak L}_{g_2}^{\bar i_2}] = {\frak L}_{g_1g_2}^{\bar i_2}.$
We can argue in a similar way from $g_1g_2, g_3$ and $g_1g_2g_3$
to get $$0 \neq [[{\frak L}_{g_1}^{\bar 0}, {\frak L}_{g_2}^{\bar
i_2}],{\frak L}_{g_2}^{\bar i_3}] ={\frak L}_{g_1g_2g_3}^{\bar i_2
+ \bar i_3}$$ for some ${\bar i_3} \in {\hu Z}_2$. Following this
process with the $\Sigma_G$-connection $\{g_1,....,g_r\}$ we
obtain that
$$0 \neq[[\cdots [[{\frak L}_{g_1}^{\bar 0}, {\frak L}_{g_2}^{\bar
i_2}],{\frak L}_{g_2}^{\bar i_3}], \cdots], {\frak L}_{g_1}^{\bar
i_r}] = {\frak L}_{g_1g_2g_3 \cdots g_r}^{\bar i_2 + \cdots +\bar
i_r}$$ and so either $0 \neq {\frak L}_{g'}^{\bar i_2 + \cdots
+\bar i_r}$ or $0 \neq {\frak L}_{(g')^{-1}}^{\bar i_2 + \cdots
+\bar i_r}.$
That is, for any $g' \in \Sigma_G \setminus
\{g,g^{-1}\}$ we have that
\begin{equation}\label{paraaqui2}
\hbox{$0 \neq {\frak L}_{\xi}^{\bar i_2 + \cdots +\bar i_r}$ for
some $\xi \in \{g', (g')^{-1}\}$.}
\end{equation}
If we argue with the $\Sigma_G$-connection $\{g_1,g_2,....,g_r\}$
as above, but starting in equation (\ref{paraaqui}) with ${\frak
L}_{g_1}^{\bar 1}$ instead of ${\frak L}_{g_1}^{\bar 0}$ we get $0
\neq {\frak L}_{\xi}^{\bar 1+ \bar i_2 + \cdots +\bar i_r}$.
Hence, and taking into account equation (\ref{paraaqui2}) we have
${\rm dim}({\frak L}_{\xi})=2$ and, by the symmetry of $\Sigma_G$,
${\rm dim}({\frak L}_{\xi^{-1}})=2$. We conclude ${\rm dim}({\frak
L}_{g'})=2$ and so ${\rm dim}({\frak L}_{g})=2$ for any $g \in
\Sigma_G$. Equation (\ref{dimen1o2}) completes the proof.
\end{proof}

\medskip

Let ${\frak L}$ be $\Sigma_G$-multiplicative, of maximal length,
with ${\mathcal Z}({\frak L}) = 0$, satisfying ${\frak L}_1 =
\sum\limits_{g \in \Sigma_G}[{\frak L}_g,{\frak L}_{g^{-1}}]$ and
with all of the elements in its $G$-support $\Sigma_G$-connected,
and consider a nonzero graded ideal $I$ of ${\frak L}$. By
equation (\ref{max}) we can write $I = (I_1^{\bar 0} \oplus
(\bigoplus\limits_{g \in \Sigma_I^{\bar 0}}{\frak L}_g^{\bar 0}))
\oplus (I_1^{\bar 1} \oplus (\bigoplus\limits_{g' \in
\Sigma_I^{\bar 1}} {\frak L}_{g'}^{\bar 1}))$ ${\rm where}
\hspace{0.1cm} \Sigma_I^{\bar i} = \{h \in \Sigma_G^{\bar i} : I
\cap {\frak L}_h^{\bar i} \neq 0\}, {\bar i} \in {\hu Z}_2$.
Furthermore,  Lemma \ref{lema4} let us assert that $\Sigma_I^{\bar
i} \neq \emptyset$ for some ${\bar i} \in {\hu Z}_2$. So, we can
take $g_0 \in \Sigma_I^{\bar i}$ such that
\begin{equation}\label{primero}
0 \neq {\frak L}_{g_0}^{\bar i} \subset I.
\end{equation}
For any $g' \in \Sigma_G \setminus \{g_0,g_0^{-1}\}$, the fact
that $g_0$ and $g'$ are $\Sigma_G$-connected gives us a
$\Sigma_G$-connection
\begin{equation}\label{cone}
\{g_1,g_2,....,g_r\}
 \end{equation}
from $g_0$ to $g'$ such that
$$g_1 = g_0,$$
$$g_1g_2, g_1g_2g_3,..., g_1g_2g_3 \cdots g_{r-1} \in \Sigma_G$$ and
$$g_1g_2g_3 \cdots g_r \in \{g',(g')^{-1}\}.$$
Consider $g_1, g_2$ and $g_1g_2$. By arguing with the
$\Sigma_G$-multiplicativity and maximal length of ${\frak L}$ as
in the proof of Lemma \ref{dimensiones}, we get
$0 \neq [{\frak L}_{g_1}^{\bar i}, {\frak L}_{g_2}^{\bar i_2}] =
{\frak L}_{g_1g_2}^{\bar i + \bar i_2}$ From here, equation
(\ref{primero}) let us conclude
$$0 \neq {\frak L}_{g_1g_2}^{\bar i + \bar i_2} \subset
I.$$ We can argue in a similar way from $g_1g_2, g_3$ and
$g_1g_2g_3$ to get $$0 \neq {\frak L}_{g_1g_2g_3}^{\bar i + \bar
i_2 + \bar i_3} \subset I$$ for some ${\bar i_3} \in {\hu Z}_2$.
Following this process with the $\Sigma_G$-connection
$\{g_1,....,g_r\}$ we obtain that
\begin{equation}\label{1003}
0 \neq {\frak L}_{g_1g_2g_3 \cdots g_r}^{\bar p} \subset I,
\end{equation}
$\bar p = \bar i + \bar i_2 + \bar i_3 +\cdots + \bar i_r$, and so
either $0 \neq {\frak L}_{g'}^{\bar p} \subset I$ or $0 \neq
{\frak L}_{(g')^{-1}}^{\bar p} \subset I$ for some ${\bar p} \in
{\hu Z}_2$. That is, for any $g' \in \Sigma_G$ we have that
\begin{equation}\label{segunda}
\hbox{$0 \neq {\frak L}_{\xi}^{\bar p} \subset I$ for some $\xi
\in \{g', (g')^{-1}\}$ and some ${\bar p} \in {\hu Z}_2$.}
\end{equation}

Taking into account Lemma \ref{dimensiones} we can distinguish two
possibilities.

In the first one ${\rm dim}({\frak L}_g)=1$ for any $g \in
\Sigma_G$ and so equation (\ref{segunda}) gives us that, in this
first possibility,
\begin{equation}\label{tododentro1}
\hbox{ either $ {\frak L}_g \subset I$ or ${\frak L}_{g^{-1}}
\subset I$ for any $g \in \Sigma_G$.}
\end{equation}

In the second possibility, ${\rm dim}({\frak L}_g)=2$ for any $g
\in \Sigma_G$. Given now any $g' \in \Sigma_G \setminus
\{g_0,g_0^{-1}\}$ and the $\Sigma_G$-connection (\ref{cone}) from
$g_0$ to $g'$, the $\Sigma_G$-multiplicativity of ${\frak L}$ let
us also get in a first step $0 \neq [{\frak L}_{g_1}^{\bar i},
{\frak L}_{g_2}^{\bar i_2 + \bar 1}] = {\frak L}_{g_1g_2}^{\bar i
+ \bar i_2+ \bar 1} \subset I$, in a second step $0 \neq [{\frak
L}_{g_1g_2}^{\bar i + \bar i_2+ \bar 1}, {\frak L}_{g_3}^{\bar
i_3}] = {\frak L}_{g_1g_2g_3}^{\bar i + \bar i_2+ \bar i_3 + \bar
1} \subset I$, and finally
$0 \neq {\frak L}_{g_1g_2g_3 \cdots g_r}^{\bar p + \bar 1} \subset
I.$ Taking into account equation (\ref{1003}), we have showed that
in this second possibility
\begin{equation}\label{otramaas}
\hbox{either ${\frak L}_{g'} \subset I$ or ${\frak L}_{(g')^{-1}}
\subset I$ for any $g' \in \Sigma_G \setminus \{g_0,g_0^{-1}\}$.}
\end{equation}
Observe that it remains to study if equation (\ref{otramaas}) also
holds for $g_0$. To do that, let us suppose the cardinal of
$\Sigma_G$, denoted by $|\Sigma_G|$, is greater than 2. Then there
exists some ${g'} \in \Sigma_G\setminus \{g_0,g_0^{-1}\}$ and, by
the above, a $\Sigma_G$-connection $\{g_0,g_2,....,g_r\}$ from
$g_0$ to $g'$ such that ${\frak L}_{g'}^{\bar 0} \oplus {\frak
L}_{g'}^{\bar 1} \subset I$ with any ${\frak L}_{g'}^{\bar i} \neq
0$ being also $0 \neq[[\cdots [[{\frak L}_{g_0}^{\bar i}, {\frak
L}_{g_2}^{\bar i_2}],{\frak L}_{g_3}^{\bar i_3}], \cdots], {\frak
L}_{g_r}^{\bar i_r}] = {\frak L}_{g'}^{\bar p}$, $\bar p = \bar i
+ \bar i_2 \cdots + \bar i_r$. From here, we also have the
$\Sigma_G$-connection $\{g_0g_2 \cdot \cdot \cdot g_r,
g_r^{-1},g_{r-1}^{-1},..., g_2^{-1}\} \subset \Sigma_G$ which
satisfies $g_0g_2 \cdot \cdot \cdot g_r, g_0g_2 \cdot \cdot \cdot
g_{r-1},..., g_0 \in \Sigma_G$ and $g_0g_2 \cdot \cdot \cdot g_r
=g'$.
  By $\Sigma_G$-multiplicativity, and taking into account ${\frak L}_{{g'}}^{\bar 0} \oplus {\frak L}_{{g'}}^{\bar
1} \subset I$ with any ${\frak L}_{{g'}}^{\bar i} \neq 0$, $0\neq
[[...[[{\frak L}_{g_0g_2 \cdot \cdot \cdot g_r}^{\bar p + \bar 1},
{\frak L}_{g_{r}^{-1}}^{\bar i_r}],{\frak L}_{g_{r-1}^{-1}}^{\bar
i_{r-1}}],...],{\frak L}_{g_{2}^{-1}}^{\bar i_2}] ={\frak
L}_{g_0}^{\bar i + \bar 1} \subset I.$ From here ${\frak L}_{g_0}
\subset I$. Let us observe that the above argument also shows
that, (under the assumption $|\Sigma_G| > 2$),
\begin{equation}\label{noparte}
\hbox{in case ${\frak L}_{g_0}^{\bar i} \subset I$ for some $g_0
\in \Sigma_G$ and some ${\bar i} \in {\hu Z}_2$, then ${\frak
L}_{g_0} \subset I$.}
\end{equation}
Summarizing the above paragraphs, equations (\ref{tododentro1}),
(\ref{otramaas}) and (\ref{noparte}) let us assert that in case
$|\Sigma_G| > 2$ then
\begin{equation}\label{casifinal}
\hbox{ either $ {\frak L}_g \subset I$ or $ {\frak L}_{g^{-1}}
\subset I$ for any $g \in \Sigma_G$.}
\end{equation}

From now on we are also going to
suppose $|\Sigma_G| > 2$, (the easy case in which $|\Sigma_G|< 2$
will be consider below in Lemma \ref{cardinal2}). Then equation
(\ref{casifinal}) let us denote by
$$\Sigma_I := \{g \in \Sigma_G : {\frak L}_g \subset I\}$$
and assert, taking into account the fact ${\frak L}_1 =
\sum\limits_{g \in \Sigma_G} [{\frak L}_g,{\frak L}_{g^{-1}}]$,
that
\begin{equation}\label{eq1}
{\frak L}_1 \subset I.
\end{equation}
Let us also denote by
$$J := \bigoplus\limits_{g' \in \Sigma_G \setminus
\Sigma_I} {\frak L}_{g'}.$$ Observe that equation (\ref{noparte}),
joint with the maximal lenght of ${\frak L}$ and the graded
character of ${\frak L}_{g'}$ and of $I$, let us assert that if
$g' \in \Sigma_G \setminus \Sigma_I$ then ${\frak L}_{g'} \cap I =
\{0\}$. From here, taking also into account equation (\ref{eq1}),
we can write
\begin{equation}\label{semeolvida}
I = {\frak L}_1 \oplus (\bigoplus\limits_{g \in \Sigma_I} {\frak
L}_g).
\end{equation}
\begin{lemma}\label{lemahafa}
The following assertions hold.
\begin{itemize}
\item[{\rm (i)}] For any $g' \in \Sigma_G \setminus \Sigma_I$ and
${\bar i} \in {\hu Z}_2$, we have $[{\frak L}_1, {\frak
L}_{g'}^{\bar i}]=0$.

\item[{\rm (ii)}] For any $g', g'' \in \Sigma_G \setminus
\Sigma_I$ and ${\bar i}, {\bar j} \in {\hu Z}_2$, we have $[{\frak
L}_{g'}^{\bar i}, {\frak L}_{g''}^{\bar j}] \subset J$.

\item[{\rm (iii)}] For any $g \in \Sigma_I, g' \in \Sigma_G
\setminus \Sigma_I$ with $g' \neq g^{-1}$ and ${\bar i}, {\bar j}
\in {\hu Z}_2$, we have $[{\frak L}_g^{\bar i}, {\frak
L}_{g'}^{\bar j}] = 0$.
\end{itemize}
\end{lemma}
\begin{proof}
(i) Suppose $[{\frak L}_1, {\frak L}_{g'}^{\bar i}] \neq 0$, then
$0 \neq [{\frak L}_1^{\bar k}, {\frak L}_{g'}^{\bar i}] = {\frak
L}_{g'}^{\bar k +\bar i}$ for some ${\bar k} \in {\hu Z}_2$.
Equation (\ref{eq1}) gives us ${\frak L}_{g'}^{\bar k +\bar i}
\subset I$ and so, taking into account equation (\ref{noparte}),
$g' \in \Sigma_I$, a contradiction. Hence $[{\frak L}_1, {\frak
L}_{g'}^{\bar i}] = 0$.

\medskip

(ii) If $[{\frak L}_{g'}^{\bar i}, {\frak L}_{g''}^{\bar j}] = 0$
then $[{\frak L}_{g'}^{\bar i}, {\frak L}_{g''}^{\bar j}] \subset
J$. So, let consider the case $[{\frak L}_{g'}^{\bar i}, {\frak
L}_{g''}^{\bar j}] \neq 0$, being then $$0 \neq [{\frak
L}_{g'}^{\bar i}, {\frak L}_{g''}^{\bar j}]={\frak
L}_{g'g''}^{\bar i + \bar j},$$ and suppose $g'g'' \in \Sigma_I$.
By $\Sigma_G$-multiplicativity, $[{\frak L}_{g'g''}^{\bar i+ \bar
j},{\frak L}_{(g'')^{-1}}^{\bar j}] = {\frak L}_{g'}^{\bar
i}\subset I$. From here, $g' \in \Sigma_I$, a contradiction. We
conclude $g'g'' \in \Sigma_G \setminus \Sigma_I$ and so $[{\frak
L}_{g'}^{\bar i}, {\frak L}_{g''}^{\bar j}] \subset J.$

\medskip

(iii) If $[{\frak L}_g^{\bar i}, {\frak L}_{g'}^{\bar j}] \neq 0$,
we have as in (ii) that $g' \in \Sigma_I$ which is a
contradiction. Hence, $[{\frak L}_g^{\bar i}, {\frak L}_{g'}^{\bar
j}] = 0$.
\end{proof}

\begin{lemma}\label{lemahamuchafa}
If $g \in \Sigma_I$ and $g^{-1} \in \Sigma_G \setminus \Sigma_I$
then $[{\frak L}_g^{\bar i}, {\frak L}_{g^{-1}}^{\bar j}] = 0$ for
any $\bar i, \bar j \in {\hu Z}_2$.

\end{lemma}
\begin{proof}
Let us consider $g \in \Sigma_I$ and $g^{-1} \in \Sigma_G
\setminus \Sigma_I$. Lemma \ref{lemahafa}-(i) gives us $[[{\frak
L}_g^{\bar i}, {\frak L}_{g^{-1}}^{\bar j}], {\frak L}_h^{\bar k}]
= 0$ in case $h \in \Sigma_G \setminus \Sigma_I$ for any $\bar i,
\bar j, \bar k \in {\hu Z}_2$. Consider now any $f \in \Sigma_I$
such that $f \neq g$ and any $\bar s \in {\hu Z}_2$, then we have
$$[[{\frak L}_g^{\bar i}, {\frak L}_{g^{-1}}^{\bar j}], {\frak
L}_f^{\bar s}] \subset [[{\frak L}_{g^{-1}}^{\bar j}, {\frak
L}_f^{\bar s}],{\frak L}_g^{\bar i}] + [[{\frak L}_f^{\bar s},
{\frak L}_g^{\bar i}], {\frak L}_{g^{-1}}^{\bar j}].$$ Taking into
account that $f \in \Sigma_I$ and that in case $[{\frak L}_f^{\bar
s}, {\frak L}_g^{\bar i}] \neq 0$ then $fg \in \Sigma_I$, Lemma
\ref{lemahafa}-(iii) shows that both above summands are zero and
so $[[{\frak L}_g^{\bar i}, {\frak L}_{g^{-1}}^{\bar j}], {\frak
L}_f^{\bar s}] = 0$. That is,
\begin{equation}\label{valecero}
\hbox{$[[{\frak L}_g^{\bar i}, {\frak L}_{g^{-1}}^{\bar j}],
{\frak L}_h^{\bar k}] = 0$ for any $h \in \Sigma_G \setminus
\{g\}$ and $\bar k \in {\hu Z}_2$.}
\end{equation}

We also have for any $l \in \Sigma_G$ and $\bar s \in {\hu Z}_2$
that
\begin{equation}\label{sumovarios}
[[[{\frak L}_g^{\bar i}, {\frak L}_{g^{-1}}^{\bar j}], {\frak
L}_g^{\bar k}], {\frak L}_l^{\bar s}] \subset [[{\frak L}_g^{\bar
k}, {\frak L}_l^{\bar s}],[{\frak L}_g^{\bar i}, {\frak
L}_{g^{-1}}^{\bar j}]]  +[[[{\frak L}_g^{\bar i}, {\frak
L}_{g^{-1}}^{\bar j}], {\frak L}_l^{\bar s}], {\frak L}_g^{\bar
k}].
\end{equation}
Consider the second summand in equation (\ref{sumovarios}). From
equation (\ref{valecero}), $$[[[{\frak L}_g^{\bar i}, {\frak
L}_{g^{-1}}^{\bar j}], {\frak L}_l^{\bar s}], {\frak L}_g^{\bar
k}] = 0$$ up to maybe for $l = g$. In this case, if $$0 \neq
[[[{\frak L}_g^{\bar i}, {\frak L}_{g^{-1}}^{\bar j}], {\frak
L}_g^{\bar s}], {\frak L}_g^{\bar k}] \subset {\frak
L}_{g^2}^{\bar i + \bar j + \bar s + \bar k},$$ being then ${\frak
L}_{g^2}^{\bar i + \bar j + \bar s + \bar k} \subset I$. We have
by $\Sigma_G$-multiplicativity $0 \neq [{\frak L}_{g^2}^{\bar i +
\bar j + \bar s + \bar k}, {\frak L}_{g^{-1}}^{\bar j}]$, but
Lemma \ref{lemahafa}-(iii) shows $[{\frak L}_{g^2}^{\bar i + \bar
j + \bar s + \bar k}, {\frak L}_{g^{-1}}^{\bar j}] = 0$, a
contradiction. Hence $$[[[{\frak L}_g^{\bar i}, {\frak
L}_{g^{-1}}^{\bar j}], {\frak L}_g^{\bar s}], {\frak L}_g^{\bar
k}] = 0,$$ and so
\begin{equation}\label{prizero}
\hbox{ $[[[{\frak L}_g^{\bar i}, {\frak L}_{g^{-1}}^{\bar j}],
{\frak L}_l^{\bar s}], {\frak L}_g^{\bar k}] = 0$ for any $l \in
\Sigma_G$.}
\end{equation}
Consider now the first summand in equation (\ref{sumovarios}). We
have $$[[{\frak L}_g^{\bar k}, {\frak L}_l^{\bar s}],[{\frak
L}_g^{\bar i}, {\frak L}_{g^{-1}}^{\bar j}]] \subset [{\frak
L}_{gl}^{\bar k + \bar s}, [{\frak L}_g^{\bar i}, {\frak
L}_{g^{-1}}^{\bar j}]].$$ But in case $l \neq g^{-1}$ we can
assert $[{\frak L}_{gl}^{\bar k + \bar s}, [{\frak L}_g^{\bar i},
{\frak L}_{g^{-1}}^{\bar j}]] = 0.$ Indeed, in the opposite case,
that is,
\begin{equation}\label{harto}
[{\frak L}_{gl}^{\bar k + \bar s}, [{\frak L}_g^{\bar i}, {\frak
L}_{g^{-1}}^{\bar j}]] \neq 0,
\end{equation}
we have that $gl \neq 1$ with $0 \neq {\frak L}_{gl}^{\bar k +
\bar s} \subset I$, and by super Jacobi identity and Lemma
\ref{lemahafa}-(iii) that $[[{\frak L}_{gl}^{\bar k + \bar s},
{\frak L}_g^{\bar i}], {\frak L}_{g^{-1}}^{\bar j}] \neq 0$, which
implies
\begin{equation}\label{vino}
[{\frak L}_{g^2l}^{\bar k + \bar s + \bar i}, {\frak
L}_{g^{-1}}^{\bar j}] \neq 0.
\end{equation}
Since by equation (\ref{harto}) and Lemma \ref{lemahafa}-(i), $l
\neq g^{-2}$ and so $g^2l \neq 1$, we have ${g^2l} \in \Sigma_I$
and $g^{-1} \in \Sigma_G \setminus \Sigma_I$. Then by Lemma
\ref{lemahafa}-(iii) we get $[{\frak L}_{g^2l}^{\bar k + \bar s +
\bar i}, {\frak L}_{g^{-1}}^{\bar j}] = 0$, which contradicts
equation (\ref{vino}). So $[{\frak L}_{gl}^{\bar k + \bar s},
[{\frak L}_g^{\bar i}, {\frak L}_{g^{-1}}^{\bar j}]] = 0$ and
hence
\begin{equation}\label{muuu}
\hbox{$[[{\frak L}_g^{\bar k}, {\frak L}_l^{\bar s}],[{\frak
L}_g^{\bar i}, {\frak L}_{g^{-1}}^{\bar j}]]=0$ if $l \neq
g^{-1}$.}
\end{equation}
Taking into account equations (\ref{sumovarios}), (\ref{prizero})
and (\ref{muuu}) we have showed that
\begin{equation}\label{porfin}
[[[{\frak L}_g^{\bar i}, {\frak L}_{g^{-1}}^{\bar j}], {\frak
L}_g^{\bar k}], {\frak L}_l^{\bar s}] = 0
\end{equation}
for any $l \in \Sigma_G \setminus \{g^{-1}\}$ and $\bar i, \bar j,
\bar k, \bar s \in {\hu Z}_2.$

Suppose now $[{\frak L}_g^{\bar i}, {\frak L}_{g^{-1}}^{\bar j}]
\neq 0$ for some $\bar i, \bar j \in {\hu Z}_2$. Since
${\mathcal{Z}}({\frak L}) = 0$ and taking into account ${\frak
L}_1 = \sum\limits_{g \in \Sigma_G}[{\frak L}_g,{\frak
L}_{g^{-1}}]$ and super Jacobi identity, there exist $h \in
\Sigma_G$ and $\bar k \in {\hu Z}_2$ such that $[[{\frak
L}_g^{\bar i}, {\frak L}_{g^{-1}}^{\bar j}], {\frak L}_h^{\bar k}]
\neq 0.$ By equation (\ref{valecero}) necessarily $h = g$ and so
\begin{equation}\label{otramas}
0 \neq [[{\frak L}_g^{\bar i}, {\frak L}_{g^{-1}}^{\bar j}],
{\frak L}_g^{\bar k}] = {\frak L}_g^{\bar i + \bar j + \bar k}.
\end{equation}
Hence, equation (\ref{porfin}) gives us $[{\frak L}_g^{\bar i +
\bar j + \bar k}, {\frak L}_l^{\bar s}] = 0$ for any $l \in
\Sigma_G \setminus \{g^{-1}\}$ and $\bar s \in {\hu Z}_2.$ That
is, we have showed that
\begin{equation}\label{casiacabo}
\hbox{in case $[{\frak L}_g^{\bar i}, {\frak L}_{g^{-1}}^{\bar j}]
\neq 0$ then $[{\frak L}_g^{\bar t}, {\frak L}_l^{\bar s}] = 0$}
\end{equation}
with ${\frak L}_g^{\bar t} \neq 0, \bar t \in {\hu Z}_2$, and for
any $l \in \Sigma_G \setminus \{g^{-1}\}$ and  any $\bar s \in
{\hu Z}_2$. However, since the cardinal of $\Sigma_G$ is greater
than 2, there exists $h \in \Sigma_G \setminus \{g, g^{-1 }\}$.
From here, we can take
$$\{g_1, g_2,..., g_r\} \subset \Sigma_G$$ a $\Sigma_G$-connection from $g$ to
$h$, (which in particular implies $g_2 \neq g_1^{-1}$),
satisfying $r \geq 2$. Then we have $g_1 = g \in \Sigma_G^{\bar
t}, g_2 \in \Sigma_G^{\bar p}$ for some $\bar p \in {\hu Z}_2$,
and $g_1g_2 \in \Sigma_G$. Hence, the $\Sigma_G$-multiplicativity
of ${\frak L}$ gives us
$$[{\frak L}_g^{\bar t}, {\frak L}_{g_2}^{\bar p}]
\neq 0$$ with $g_2 \neq g^{-1}$, a contradiction with equation
(\ref{casiacabo}). Hence $[{\frak L}_g^{\bar i}, {\frak
L}_{g^{-1}}^{\bar j}] = 0$ for any $\bar i, \bar j \in {\hu Z}_2$
and the proof is complete.
\end{proof}

\begin{lemma}\label{lema_util}
If $g, g^{-1} \in \Sigma_I$ for some $g \in \Sigma_G$ then $I =
{\frak L}$.
\end{lemma}
\begin{proof}
Given any $g' \in \Sigma_G \setminus \{g, g^{-1}\}$ there exists a
$\Sigma_G$-connection $\{g_1,...,g_n\}$ with $g_1 = g'$ and $g_1
\cdots g_n \in \{g,g^{-1}\}$. From here, we can consider the
$\Sigma_G$-connection $\{g_1 \cdots g_n,g_n^{-1},...,g_2^{-1}\}$
from $g_1 \cdots g_n$ to $g$. By Lemma \ref{dimensiones} we can
distinguish two possibilities. In the first one ${\rm dim}({\frak
L}_h) = 1$ for any $h \in \Sigma_G$ and so ${\frak L}_{g_1 \cdots
g_n} = {\frak L}_{g_1 \cdots g_n}^{\bar i}$ for some $\bar i \in
{\hu Z}_2$. The $\Sigma_G$-multiplicativity of ${\frak L}$ let us
assert
$$[[...[[{\frak L}_{g_1 \cdots g_n}^{\bar i}, {\frak L}_{g_n^{-1}}^{\bar i_1}], {\frak
L}_{g_{n-1}^{-1}}^{\bar i_2}],...], {\frak L}^{\bar
i_n}_{{g_2}^{-1}}] = {\frak L}_{g'}^{\bar j} = {\frak L}_{g'},$$
for some $\bar i_1, \bar i_2, ..., \bar i_n, \bar j \in {\hu
Z}_2$. Since ${\frak L}_g, {\frak L}_{g^{-1}} \subset I$ and so
${\frak L}_{g_1 \cdots g_n}^{\bar i} \subset I$ we get ${\frak
L}_{g'} \subset I$. This fact joint with equation (\ref{eq1}) let
us conclude $I = {\frak L}$. In the second possibility, ${\rm
dim}({\frak L}_h) = 2$ for any $h \in \Sigma_G$. The
$\Sigma_G$-multiplicativity of ${\frak L}$ let us now assert
$$[[...[[{\frak L}_{g_1 \cdots g_n}^{\bar 0}, {\frak L}_{g_n^{-1}}^{\bar i_1}], {\frak
L}_{g_{n-1}^{-1}}^{\bar i_2}],...], {\frak L}^{\bar
i_n}_{{g_2}^{-1}}] = {\frak L}_{g'}^{\bar j},$$ for some $\bar
i_1, \bar i_2, ..., \bar i_n, \bar j \in {\hu Z}_2$. As above, the
fact ${\frak L}_g, {\frak L}_{g^{-1}} \subset I$ implies ${\frak
L}_{g_1 \cdots g_n}^{\bar 0} \subset I$ and then ${\frak
L}_{g'}^{\bar j} \subset I$. In a similar way we have
$$[[...[[{\frak L}_{g_1 \cdots g_n}^{\bar 1}, {\frak
L}_{g_n^{-1}}^{\bar i_1}], {\frak L}_{g_{n-1}^{-1}}^{\bar
i_2}],...],{\frak L}^{\bar i_n}_{{g_2}^{-1}}] = {\frak
L}_{g'}^{\bar j+ \bar 1},$$ and so ${\frak L}_{g'}^{\bar j + \bar
1} \subset I$. Hence ${\frak L}_{g'} \subset I$ and, taking also
into account equation (\ref{eq1}), $I = {\frak L}$.
\end{proof}

\begin{theorem}\label{lema4I}
Let ${\frak L}$ be of maximal length, $\Sigma_G$-multiplicative,
with ${\mathcal{Z}}({\frak L}) = 0$ and satisfying ${\frak L}_1 =
\sum\limits_{g \in \Sigma_G}[{\frak L}_g,{\frak L}_{g^{-1}}]$. If
$|\Sigma_G| > 2$ and $\Sigma_G$ has all of its elements
$\Sigma_G$-connected, then the following assertion hold.
\begin{enumerate}
\item Either ${\frak L}$ is gr-simple or

\item ${\frak L} = I \oplus J$ with $I, J$ gr-simple graded ideals
of ${\frak L}$ satisfying $[I,J] = 0$.
\end{enumerate}
\end{theorem}

\begin{proof}
Consider any nonzero graded ideal $I$ of ${\frak L}$. Observe that
from equation (\ref{semeolvida}) we can write $$I = {\frak L}_1
\oplus (\bigoplus\limits_{g \in \Sigma_I} {\frak L}_g)$$ where
$\Sigma_I := \{g \in \Sigma_G : {\frak L}_g \subset I\}$.

Suppose $\Sigma_I \neq \Sigma_G$. Denote as above
\begin{equation}\label{J}
J := \bigoplus \limits_{g' \in \Sigma_G \setminus \Sigma_I} {\frak
L}_{g'},
\end{equation}
being so $J \neq \{0\}$. Let us show that the graded subspace $J$
is a graded ideal of ${\frak L}$. We can write
$$[J,{\frak L}] = [\bigoplus\limits_{g \in
\Sigma_G\setminus \Sigma_I}{\frak L}_g, {\frak L}_1 \oplus
(\bigoplus\limits_{g' \in \Sigma_I}{\frak L}_{g'}) \oplus
(\bigoplus\limits_{g'' \in \Sigma_G \setminus \Sigma_I} {\frak
L}_{g''})] \subset$$
\begin{equation}\label{summands}
[\bigoplus\limits_{g \in \Sigma_G\setminus \Sigma_I}{\frak L}_g,
{\frak L}_1] + [\bigoplus\limits_{g\in \Sigma_G\setminus
\Sigma_I}{\frak L}_g, \bigoplus\limits_{g' \in \Sigma_I}{\frak
L}_{g'}] + [\bigoplus\limits_{g\in \Sigma_G\setminus
\Sigma_I}{\frak L}_g, \bigoplus\limits_{g'' \in \Sigma_G \setminus
\Sigma_I} {\frak L}_{g''}].
\end{equation}
Taking into account Lemma \ref{lemahafa}-(i), we get that the
first summand in equation (\ref{summands}) vanishes. Let us also
observe that Lemma \ref{lemahafa}-(iii) and Lemma
\ref{lemahamuchafa} give us $[\bigoplus\limits_{g\in
\Sigma_G\setminus \Sigma_I}{\frak L}_g, \bigoplus\limits_{g' \in
\Sigma_I}{\frak L}_{g'}] = 0$. That is, the second summand in
equation (\ref{summands}) is also null and so we have $[J,{\frak
L}] \subset [\bigoplus\limits_{g\in \Sigma_G\setminus
\Sigma_I}{\frak L}_g, \bigoplus\limits_{g'' \in \Sigma_G \setminus
\Sigma_I} {\frak L}_{g''}].$ Lemma \ref{lemahafa}-(ii) finally
shows $[J,{\frak L}] \subset J$ and we have that $J$ is a graded
ideal of ${\frak L}$.

Let us distinguish two cases. In the first one ${\frak L}_1 \neq
\{0\}$. Then equation (\ref{eq1}) gives us ${\frak L}_1 \subset J$
which contradicts equation (\ref{J}). Hence $\Sigma_I = \Sigma_G$
for any graded ideal $I$ of ${\frak L}$ and we conclude ${\frak
L}$ is gr-simple in the case ${\frak L}_1 \neq \{0\}$.

In the second one, ${\frak L}_1 = \{0\}$. We have showed above
that we can write
\begin{equation}\label{dosober}
\hbox{${\frak L} = I \oplus J$ with $J$ a nonzero graded ideal of
${\frak L}$ satisfying $[I,J] = 0$.}
\end{equation}
Now, we can prove the gr-simplicity of $I = {\frak L}_1 \oplus
(\bigoplus\limits_{g \in \Sigma_I} {\frak L}_g)$ by observing that
the $\Sigma_G$-multiplicativity of ${\frak L}$ and Lemma
\ref{lemahafa}-(iii) give us that $\Sigma_I$ has all of its
elements $\Sigma_I$-connected, that is, $\Sigma_G$-connected
through $\Sigma_G$-connections contained in $\Sigma_I$, and that
$I$ is $\Sigma_I$-multiplicative. We also have ${\mathcal Z}_I(I)
= 0, ({\mathcal Z}_I(I):= \{x \in I: [x,I] = 0\}$), as consequence
of equation (\ref{dosober}) and ${\mathcal Z}({\frak L}) = 0$, and
we clearly have $\dim {\frak L}_g = 1$ for any $g \in \Sigma_I$.
Since Lemma \ref{lema_util} gives us that in case $g \in \Sigma_I$
then necessarily $g^{-1} \notin \Sigma_I$, if we consider a
nonzero graded ideal ${\widetilde I}$ of $I$ we can argue with the
fact that $\Sigma_I$ has all of its elements $\Sigma_I$-connected,
the $\Sigma_I$-multiplicativity and the maximal length of $I$ as
usual to get ${\widetilde I} = I$.
So $I$ is $gr$-simple. The same argument applies to show $J$ is
also a $gr$-simple graded ideal of ${\frak L}$.
\end{proof}


\medskip

It remains to study the case in which $|\Sigma_G|\leq 2$. We note
that it is straightforward to describe this class of graded Lie
superalgebras in a much more detailed way than we do in Lemma
\ref{cardinal2}. However, it is better for our purposes the
compact description given below.

\begin{lemma}\label{cardinal2}
Let ${\frak L}$ be  of maximal length, with ${\mathcal{Z}}({\frak
L}) = 0$ and satisfying ${\frak L}_1 = \sum\limits_{g \in
\Sigma_G}[{\frak L}_g,{\frak L}_{g^{-1}}]$. If $|\Sigma_G| \leq
2$, then one of the following assertions hold.

\begin{enumerate} \item[{\rm (1).}] ${\frak L}$ is a gr-simple Lie
superalgebra.

\item[{\rm (2).}] $|\Sigma_G| = 1$ and ${\frak L}$ is a
$(2+n)$-dimensional non gr-simple Lie superalgebra with $n \in
\{1,2,3\}$, graded as ${\frak L} = [{\frak L}_g,{\frak L}_g]
\oplus {\frak L}_g$ with ${\rm dim}({\frak L}_g) = 2$ and ${\rm
dim}([{\frak L}_g,{\frak L}_g]) = n$, for some $g \in G \setminus
\{1\}$ such that $g = g^{-1}$.

\item[{\rm (3).}] $|\Sigma_G| = 2$ and ${\frak L}$ is a
$3$-dimensional non gr-simple Lie superalgebra which decomposes as
${\frak L} = I \oplus J$, with $I=[{\frak L}_g,{\frak L}_{g^{-1}}]
\oplus {\frak L}_g$ a $2$-dimensional graded ideal of ${\frak L}$
and $J={\frak L}_{g^{-1}}$, for some $g \in G \setminus \{1\}$
such that $g \neq g^{-1}$.

\item[{\rm (4).}] $|\Sigma_G| = 2$ and ${\frak L}$ is a
$(4+n)$-dimensional non gr-simple Lie superalgebra with $n \in
\{0,1,2,3,4\}$, graded as ${\frak L} = [{\frak L}_g,{\frak
L}_{g^{-1}}] \oplus {\frak L}_g \oplus {\frak L}_{g^{-1}}$ with
${\rm dim}({\frak L}_g) = {\rm dim}({\frak L}_{g^{-1}}) = 2$ and
${\rm dim}([{\frak L}_g,{\frak L}_{g^{-1}}]) = n$, for some $g \in
G \setminus \{1\}$ such that $g \neq g^{-1}$.
\end{enumerate}
\end{lemma}

\begin{proof}
Suppose ${\frak L}$ is not gr-simple, then there exists a nonzero
graded ideal $I$ of ${\frak L}$ such that $I \neq {\frak L}$.

In the case $|\Sigma_G|=1$, we have that $\Sigma_G = \{g\}$ for
some $g \in G \setminus \{1\}$ such that $g = g^{-1}$. By Lemma
\ref{lema4}, equation (\ref{idealpartio}) and the maximal length
of ${\frak L}$ we have $0 \neq {\frak L}_g^{\bar i} \subset I$ for
some ${\bar i} \in {\hu Z}_2$. By the maximal length of ${\frak
L}$, see equation (\ref{dimen1o2}), we have two cases to consider.
In the first one ${\rm dim}({\frak L}_g) = 1$, being then ${\frak
L}_g ={\frak L}_g^{\bar i}$. In this case we can write ${\frak L}
= [{\frak L}_g^{\bar i},{\frak L}_g^{\bar i}] \oplus {\frak
L}_g^{\bar i}$. Since $0 \neq {\frak L}_g^{\bar i} \subset I$,
then $I = {\frak L}$, (being also $[{\frak L}_g^{\bar i}, {\frak
L}_g^{\bar i}] \neq 0$ as consequence of ${\mathcal{Z}}({\frak L})
= 0$), a contradiction. From here, necessarily the second case to
consider, in which ${\rm dim}({\frak L}_g) = 2$, happens. This
case clearly gives us item (2).

In the case $|\Sigma_G| = 2$, we have $\Sigma_G = \{g, g^{-1}\}$
for some $g \in G$ such that $g \neq g^{-1}$. We also have by
Lemma \ref{lema4}, equation (\ref{idealpartio}) and the maximal
length of ${\frak L}$ that either $0 \neq {\frak L}_g^{\bar i}
\subset I$ or $0 \neq {\frak L}_{g^{-1}}^{\bar i} \subset I$ for
some ${\bar i} \in {\hu Z}_2$. We can suppose without any loss of
generality that $0 \neq {\frak L}_g^{\bar i} \subset I$. By the
maximal length of ${\frak L}$, see equation (\ref{dimen1o2}), we
have again two cases to consider. In the first one ${\rm
dim}({\frak L}_g) = {\rm dim}({\frak L}_{g^{-1}}) = 1$, being then
${\frak L}_g = {\frak L}_g^{\bar i}$ and ${\frak L}_{g^{-1}} =
{\frak L}_{g^{-1}}^{\bar i}$ for some $\bar i \in {\hu Z}_2$. In
this case we can write ${\frak L} = [{\frak L}_g, {\frak
L}_{g^{-1}}] \oplus {\frak L}_g \oplus {\frak L}_{g^{-1}}$. Since
$0 \neq {\frak L}_g \subset I$, then $[{\frak L}_g, {\frak
L}_{g^{-1}}] \subset I$. As we also have $[{\frak L}_g, {\frak
L}_g] \subset {\frak L}_{g^2}$, necessarily $[{\frak L}_g, {\frak
L}_g] = 0$ or $[{\frak L}_g,{\frak L}_g] \subset {\frak
L}_{g^{-1}}$. From here, if $[{\frak L}_g,{\frak L}_g] \neq 0$,
then $0 \neq [{\frak L}_g,{\frak L}_g] = {\frak L}_{g^{-1}}
\subset I$ which would imply $I = {\frak L}$, a contradiction.
Hence $[{\frak L}_g,{\frak L}_g] = 0.$ Taking into account
${\mathcal{Z}}({\frak L}) = 0$, we get $[{\frak L}_g, {\frak
L}_{g^{-1}}] \neq 0$. From here ${\rm dim}({\frak L}) = 3$ and
${\rm dim}({I}) = 2$ being $I=[{\frak L}_g, {\frak L}_{g^{-1}}]
\oplus {\frak L}_g$. Moreover, $[[{\frak L}_g, {\frak
L}_{g^{-1}}],{\frak L}_{g^{-1}}] = 0$ since in the opposite case
${\frak L}_{g^{-1}} \subset I$ and so $I = {\frak L}$. This is
item (3) in the lemma. Finally, the second case to consider is the
one satisfying ${\rm dim}({\frak L}_g) = {\rm dim}({\frak
L}_{g^{-1}}) = 2$, which clearly gives rise to item (4).
\end{proof}


\begin{theorem}\label{teo4}
Let ${\frak L}$ be of maximal length, $\Sigma_G$-multiplicative,
with ${\mathcal{Z}}({\frak L}) = 0$ and satisfying ${\frak L}_1 =
\sum\limits_{g \in \Sigma_G}[{\frak L}_g,{\frak L}_{g^{-1}}]$.
Then $${\frak L} =(\bigoplus\limits_{k \in K} I_k) \oplus
(\bigoplus\limits_{q \in Q} I_q),$$ where any $I_k$ is a gr-simple
graded ideal of ${\frak L}$ having its $G$-support,
$\Sigma_{I_k}$, with all of the elements $\Sigma_{I_k}$-connected,
and any $I_q$ is one of the graded ideals (2)-(4) in Lemma
\ref{cardinal2} satisfying $[I_q,I_{q'}] = 0$ for any $q'\in Q$
with $q \neq q'$.
\end{theorem}

\begin{proof}
By Corollary \ref{co1}, ${\frak L} = \bigoplus\limits_{[g] \in
\Sigma_G/\sim} I_{[g]}$ is the direct sum of the graded ideals
$$I_{[g]} = {\frak L}_{{\frak C}_g,1} \oplus {V}_{{\frak C}_g}
= (\sum\limits_{g' \in {\frak C}_g}[{\frak L}_{g'}, {\frak
L}_{(g')^{-1}}]) \oplus (\bigoplus \limits_{g' \in {\frak C}_g}
{\frak L}_{g'}),$$ having any $I_{[g]}$ as $G$-support,
$\Sigma_{I_{[g]}}= [g]$. Observe that we can write
\begin{equation}\label{rompeendos}
{\frak L} =(\bigoplus\limits_{{\tiny \begin{array}{c}
  [h] \in \Sigma_G/\sim; \\
  |[h]| > 2 \\
\end{array}}} I_{[h]})
\oplus (\bigoplus\limits_{{\tiny \begin{array}{c}
   [h'] \in \Sigma_G/\sim;\\
   |[h']| \leq 2 \\
\end{array}}}
{I}_{[h']}).
\end{equation}
In order to apply Theorem \ref{lema4I} and Lemma \ref{cardinal2}
to each $I_{[h]}$ and $I_{[h']}$ respectively, observe that any
$I_{[g]}, [g] \in \Sigma_G/\sim$, is clearly of maximal length and
that ${\mathcal Z}_{I_{[g]}}(I_{[g]}) = 0$ as consequence of
$[I_{[g]},I_{[g']}] = 0$ if $[g] \neq [g']$, (Corollary
\ref{co1}), and ${\mathcal Z}({\frak L}) = 0$. Hence, if $|[g]|
\leq 2$ then Lemma \ref{cardinal2} gives us either $I_{[g]}$ is
gr-simple or is one of the graded ideals (2)-(4). In this context,
let us denote by
$$\hbox{$K_1=\{[g] \in \Sigma_G/\sim: |[g]|\leq 2$ and $I_{[g]}$ is
gr-simple $\}$}$$ and by
$$\hbox{$Q=\{[g] \in \Sigma_G/\sim: |[g]|\leq 2$ and $I_{[g]}$ is
one of the ideals (2)-(4) in Lemma \ref{cardinal2} $\}$}.$$ Then
we can write
\begin{equation}\label{penultimaaaa}
\bigoplus\limits_{{\tiny \begin{array}{c}
   [h'] \in \Sigma_G/\sim;\\
   |[h']| \leq 2 \\
\end{array}}}
{I}_{[h']} =(\bigoplus\limits_{[h^{\prime}] \in K_1} {I}_{[h']})
\oplus (\bigoplus\limits_{[h^{\prime}] \in Q} {I}_{[h']}).
\end{equation}
If $|\Sigma_{I_{[g]}}|= |[g]| > 2$, the
$\Sigma_G$-multiplicativity of ${\frak L}$ and Lemma
\ref{lemahafa}-(iii) show that $\Sigma_{I_{[g]}}$ has all of its
elements $\Sigma_{I_{[g]}}$-connected, that is,
$\Sigma_G$-connected through $\Sigma_G$-connections contained in
$\Sigma_{I_{[g]}}$, and that $I_{[g]}$ is
$\Sigma_{I_{[g]}}$-multiplicative. From here, Theorem \ref{lema4I}
let us deduce that either $I_{[g]}$ is gr-simple or $I_{[g]} =
I_{I_{[g]}} \oplus J_{I_{[g]}}$ with $I_{I_{[g]}}, J_{I_{[g]}}$
gr-simple graded ideals of $I_{[g]}$ satisfying
$[I_{I_{[g]}},J_{I_{[g]}}] = 0$. Now let us denote by
$$\hbox{$K_2 = \{[g] \in \Sigma_G/\sim: |[g]|> 2$ and $I_{[g]}$ is
gr-simple $\}$}$$ and by
$$\hbox{$K_3 = \{[g] \in \Sigma_G/\sim: |[g]|> 2$ and $I_{[g]} =
I_{I_{[g]}} \oplus J_{I_{[g]}}$ with $I_{I_{[g]}}, J_{I_{[g]}}$
gr-simple $\}$}.$$ We can assert
\begin{equation}\label{ultimaaaa}
\bigoplus\limits_{{\tiny \begin{array}{c}
   [h] \in \Sigma_G/\sim;\\
   |[h]| > 2 \\
\end{array}}}
{I}_{[h]} =(\bigoplus\limits_{[h] \in K_2} {I}_{[h]}) \oplus
(\bigoplus\limits_{[h] \in K_3} I_{I_{[h]}})\oplus
(\bigoplus\limits_{[h] \in K_3} J_{I_{[h]}}).
\end{equation}
Taking now into account equations (\ref{rompeendos}),
(\ref{penultimaaaa}) and (\ref{ultimaaaa}), we get the
decomposition
$${\frak L} = (\bigoplus\limits_{k \in K} I_k)
\oplus (\bigoplus\limits_{q \in Q} I_q),$$ where any $I_k$ is a
gr-simple graded ideal of ${\frak L}$ and any $I_q$ is one of the
graded ideals (2)-(4) in Lemma \ref{cardinal2}, satisfying
$[I_q,I_{q'}] = 0$ for any $q'\in Q$ with $q \neq q'$.
\end{proof}

\end{document}